\documentclass{amsart}

\usepackage[utf8]{inputenc}
\usepackage{amsfonts, amsmath, amssymb, amsthm, amscd, xcolor}
\usepackage{hyperref}
\usepackage{graphicx}
\usepackage{amssymb}
\usepackage{mathtools}
\usepackage{anysize}
\marginsize{2.5cm}{2.5cm}{2.5cm}{2.5cm}

\newcommand{\R}{\mathbb R}
\newcommand{\Pp}{\mathcal P}
\renewcommand{\P}{\mathcal P}
\newcommand{\C}{\mathcal C}

\newcommand{\abs}[1]{\lvert#1\rvert}
\newtheorem{thm}{Theorem}
\newtheorem{proposition}{Proposition}
\newtheorem{lemma}{Lemma}
\newtheorem{corollary}{Corollary}

\theoremstyle{remark} 
\newtheorem{remark}{Remark}

\DeclareMathOperator{\vol}{vol}
\DeclareMathOperator{\area}{area}
\DeclareMathOperator{\sign}{sign}
\DeclareMathOperator{\conv}{conv}

\newcommand{\eps}{\omega}
\newcommand{\la}{\lambda}

\title{Equipartitions and Mahler volumes of symmetric convex bodies}
\author[Fradelizi, Hubard, Meyer, Rold\'an-Pensado, Zvavitch]{Matthieu Fradelizi, Alfredo Hubard,  Mathieu Meyer, Edgardo Rold\'an-Pensado and Artem Zvavitch}

\thanks{AH was partially supported by the ANR-17-CE40-0033 (SoS) and ANR-17-CE40-0018 (CAAPS). ERP was partially supported by PAPIIT project IA102118 and CONACyT project 282280.  MF was partially supported by  the Agence Nationale de la Recherche, project GeMeCoD (ANR 2011 BS01 007 01).  AZ was partially supported by the U.S. National Science Foundation Grant DMS-1101636 and by La Comue Universit\'e Paris-Est and B\'ezout Labex of Universit\'e Paris-Est funded by ANR, reference ANR-10-LABX-58.  AH and ERP thank Casa M\'atematica Oaxaca and \emph{Mat\'ematicos Mexicanos en el Mundo}.} 


\begin{document}

\begin{abstract}
Following ideas of Iriyeh and Shibata we give a short proof of the three-dimensional Mahler conjecture  for symmetric convex bodies. 
Our contributions include, in particular, simple self-contained proofs of their two key statements. The first of these is an equipartition (ham sandwich type) theorem which refines a celebrated result of Hadwiger and, as usual, can be proved using ideas from equivariant topology. The second is an inequality relating the product volume to areas of certain sections and their duals. Finally we give an alternative proof of the characterization of convex bodies that achieve the equality case and establish a  new stability result. 
\end{abstract}

\maketitle

\section{Introduction}

A {\it convex body} is a compact convex subset of $\R^n$ with non empty interior.  We say that $K$ is {\it symmetric} if it is centrally symmetric with its center at the origin, i.e. $K=-K$.   We write $|L|$ for the $k$-dimensional Lebesgue measure (volume)  of a measurable set $L\subset \R^n$, where $k$ is the dimension of the minimal affine subspace containing $L$. We will refer to \cite{AGM, Sc} for general references for convex bodies.
The polar body $K^\circ$ of a symmetric convex body $K$  is defined by
$K^\circ=\{y; \langle x,y\rangle\le1, \forall x\in K\}$
and its \emph{volume product} by
$$
\Pp(K) = |K| |K^\circ|.
$$
It is a linear invariant of $K$, that is
$\Pp(TK)=\Pp(K)$ for every linear isomorphism $T: \R^n \rightarrow\R^n$.  The set of convex symmetric bodies in $\R^n$ wih the Banach-Mazur distance is compact and the product volume $K\mapsto \Pp(K)$ is a continuous function (see the definition in Section \ref{sec:stability} below) hence the maximum and minimum values of $\Pp(K)$ are attained.
The Blaschke-Santal\'o inequality states that
$$
\Pp(K) \le \Pp(B^n_2),
$$
where $B^n_2$ is the Euclidean unit ball. Moreover the previous inequality is an equality if and only if $K$ is an ellipsoid (\cite{San1949}, \cite{Pet}, see \cite{MP}
or also \cite{MR2} for a simple proof of both the inequality and
the case of equality).
In \cite{Mah1939} Mahler conjectured that for every symmetric convex body $K$ in $\R^n$,
$$
\Pp(K)\ge \Pp(B_\infty^n)=\frac{4^n}{n!},
$$
where $B_\infty^n=[-1;1]^n$ and proved it for $n=2$ \cite{Mah1939m}. Later the conjecture was proved for unconditional convex bodies \cite{SR,Meyer}, zonoids \cite{Reisner, GMR} and other special cases 
\cite{Kar19, Barthe, FMZ}. The conjecture was proved in \cite{BourMil} up to a multiplicative $c^n$ factor for some constant $c>0$, see also \cite{Kuperberg, Na, Giannop}. It is also known that the cube, its dual $B_1^n$ and more generally Hanner polytopes are local minimizers \cite{NazZva, Kim} and that the conjecture follows from conjectures in systolic geometry \cite{ABT} and symplectic geometry \cite{AKY,Kar19}. 

Iriyeh and Shibata \cite{IS} came up with a beautiful proof of this conjecture in dimension $3$ that generalizes a proof of Meyer \cite{Meyer} in the unconditional case by adding two new ingredients: differential geometry and a ham sandwich type (or equipartition) result. In this mostly self-contained note we provide an alternative proof and derive the three dimensional symmetric Mahler conjecture following their work.

\begin{thm}[\cite{IS}] \label{thm:main}
For every convex symmetric body $K$ in $\R^3$,
$$\abs{K}\abs{K^\circ} \geq  \abs{B_1^3}\abs{B_\infty^3}=\frac{32}{3}.$$
Equality is achieved if and only if $K$ or $K^{\circ} $ is a parallelepiped.
\end{thm}

In Section \ref{sec:equipart} we prove an equipartition result. In Section \ref{sec:main} we derive the key inequality and put it together with the aforementioned equipartition to prove Theorem \ref{thm:main}. In Section \ref{sec:equality} we prove the equality cases of Theorem \ref{thm:main} and in Section  \ref{sec:stability} we present a new corresponding stability result. \\

\paragraph{{\bf{Acknowledgements}}} 
We thank  Shlomo Reisner for his observations on the paper of Iriyeh and Shibata \cite{IS} and Pavle Blagojevi\'c, Roman Karasev and Ihab Sabik for comments on an earlier version of this work. 

\section{An equipartition result}\label{sec:equipart}

A celebrated result of Hadwiger \cite{Hadwiger} who answered a question of Gr\"unbaum \cite{Grunbaum} shows that for any absolutely continuous finite measure in $\R^3$ there exists three planes for which any octant has $\frac 1 8 $ of the measure. There is a vast literature around Hadwiger's theorem, see \cite{FDEP, Ramos, MVZ,Makeev2007,  BlaZie18}. Theorem \ref{thm:equipart} below which corresponds to formula (15) in \cite{IS} refines Hadwiger's theorem when the measure is centrally symmetric in a way that is reminiscent of the spicy chicken theorem \cite{kha,AKK}. 

\begin{thm}\label{thm:equipart}
Let $K \subset \R^3$ be a  symmetric convex body. Then there exist planes $H_1,H_2,H_3$ passing through the origin such that:
\begin{itemize}
	\item they split $K$ into $8$ pieces of equal volume, and
	\item for each plane $H_i$, the section $K \cap H_i$ is split into $4$ parts of equal area by the other two planes.
\end{itemize}
\end{thm}

Notice that in the proof of this theorem, the convexity of $K$ is not used. The convex body $K$ could be replaced by a  symmetric measure defined via a density function and a different  symmetric density function could be used to measure the areas of the sections. \\ \begin{proof}[Proof of Theorem \ref{thm:equipart}]
The scheme of this proof is classical in applications of algebraic topology to discrete geometry. It is often referred to as the configuration-space/test-map scheme (see e.g. Chapter 14 in \cite{TOG2017}). 
Assume that $H\subset\R^3$ is an oriented plane with outer normal $v$. Let us denote the halfspaces 
$H^+=\{x; \langle x,v\rangle> 0\}$ and $H^-=\{x; \langle x,v\rangle< 0\}$. If $u\in H^+$, we say that $u$ is on the positive side of $H$.

Given the convex body $K\subset\R^3$, we parametrize a special family of triplets of planes by orthonormal bases $U=(u_1,u_2,u_3)\in SO(3)$ in the following way.

Let $H_1$ be the  plane $u_1^\perp=\{x; \langle x,u_1\rangle=0\}$.   
Let $l_2,l_3\subset H_1$ be the unique pair of  lines through the origin (as in the left part of Figure \ref{fig}) with the following properties: 
\begin{itemize}
	\item $u_2,u_3$ are directed along the angle bisectors of $l_2$ and $l_3$,
	\item  the lines $l_2$ and $l_3$ split $H_1\cap K$ into four regions of equal area,
	\item  For any $x \in l_3$,  $\langle x, u_2 \rangle$ has the same  sign  as  $\langle x, u_3\rangle$.
	
\end{itemize}
For $i=2,3$ let  $H_i$  be   the unique  plane  containing $l_i$ that splits $K\cap H_1^+$ into two parts of equal volume  and let   $H_i^+$ be the half-space limited by $H_i$  which contains  $u_2$. Thus
\begin{itemize}
\item  $\abs{K\cap H_1  \cap (H_2^+ \cap H_3^+)}=\frac{1}{4}\abs{K \cap H_1}$,
\item  $\abs{K \cap (H_1^+ \cap H_i^+)}=\frac{1}{2}\abs{K \cap H_1^+} =\frac{1}{4}\abs{K}$ for $i=2,3$.
\end{itemize}
By using standard arguments it can be seen that the lines $\{l_i\}_{i=2,3}$, the planes  $\{H_i\}_{i=1,2,3}$ and  the half-spaces $\{H_i^+\}_{i=1,2,3}$  are uniquely determined and depend continuously on $U=(u_1,u_2,u_3)$.

\begin{figure}
\includegraphics[width=\textwidth]{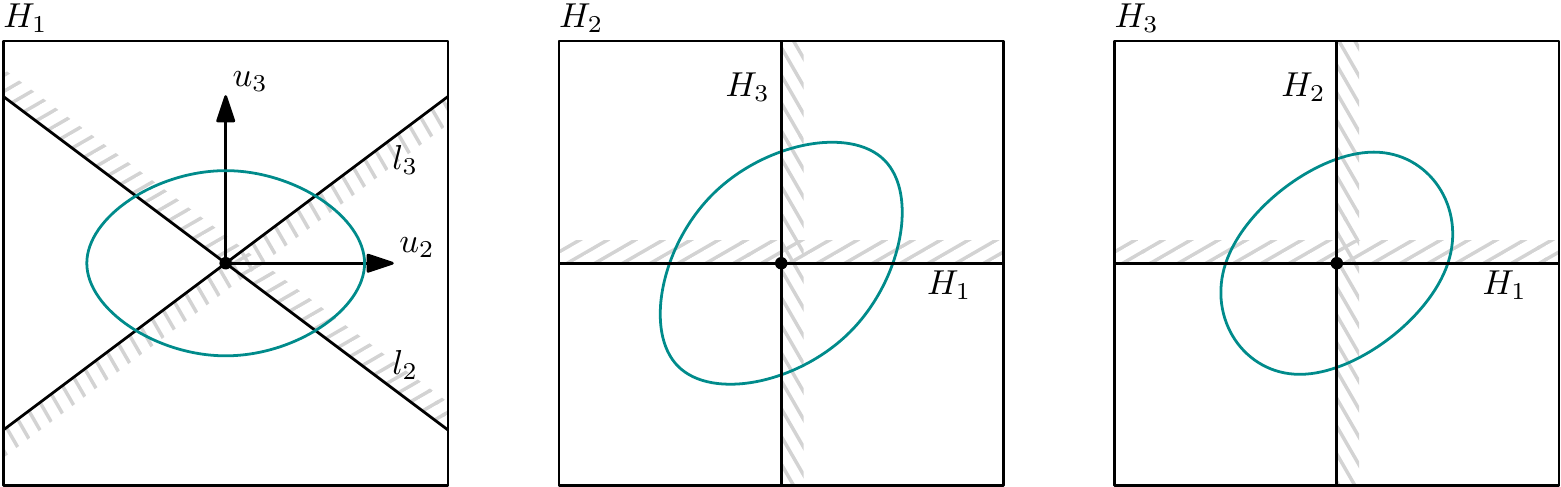}
\caption{The main parts of our construction restricted to the planes $H_1$, $H_2$ and $H_3$. Gray marks are used on the positive sides of oriented lines and planes. In the middle and right figures the horizontal lines coincide with $l_2$ and $l_3$, respectively.}
\label{fig}
\end{figure}

Now we  define a continuous test-map $F_1=(f_1, f_2, f_3):SO(3)\to\R^3$, by
\begin{align*}
	f_1(U)&=\frac 18\vol(K) - \vol(K\cap H_1^+\cap H_2^+\cap H_3^+),\\
	f_2(U)&=\frac 14\area(K\cap H_2) - \area(K\cap H_2\cap H_1^+\cap H_3^+),\\
	f_3(U)&=\frac 14\area(K\cap H_3) - \area(K\cap H_3\cap H_1^+\cap H_2^+),
\end{align*}
where $H_i = H_i(U)$, for $i=1,2,3$. Clearly, any zero of $F_1$ corresponds to a partition with the desired properties.

The dihedral group $D_4$ of eight elements with generators $g_1$ of order two and $g_2$ of order four, acts freely on $SO(3)$ by
\begin{align*}
g_1 \cdot (u_1,u_2,u_3)&=(-u_1,u_2,-u_3),\\
g_2\cdot(u_1,u_2,u_3)&=(u_1,u_3,-u_2).
\end{align*}
It also acts on $\R^3$ linearly by
\begin{align*}
g_1\cdot(x_1,x_2,x_3)&=(-x_1,-x_3,-x_2),\\
g_2\cdot(x_1,x_2,x_3)&=(-x_1,x_3,-x_2).
\end{align*}
Since $K$ is symmetric, $F_1$ is $D_4$-equivariant under the actions we just described, {\em i.e.} 
\begin{equation}\label{eq:zho}
g_i\cdot F_1(U)=F_1(g_i\cdot U),\quad\hbox{for all $U=(u_1,u_2,u_3)\in SO(3)$}.
\end{equation}
Indeed, observe that $g_1$ and $g_2$ transform $(H_1^+,H_2^+,H_3^+)$ into $(H_1^-,H_3^+,H_2^+)$ and $(H_1^+,H_3^-,H_2^+)$, respectively (see Figure \ref{fig}). Next, to establish  (\ref{eq:zho}),  observe that since $K$ is  symmetric, the volume of a set of the form $K\cap H_1^\pm\cap H_2^\pm\cap H_3^\pm$ can only have two possible values which depend only on the parity of the number of positive half-spaces used. The same is true for the area of a set of the form $K \cap H_2\cap H_1^\pm\cap H_3^\pm$ and $K \cap H_3\cap H_1^\pm\cap H_2^\pm$. 

Consider the function $F_0: SO(3) \to \R^3$
\[F_0(U)=F_0(u_1,u_2,u_3)=
\begin{pmatrix}
u_{2,2} u_{3,2}\\
u_{3,1} + u_{2,1}\\
u_{3,1} - u_{2,1}
\end{pmatrix},
\]
where $u_{i,j}$ represents the $j$-entry of the vector $u_i$. A direct computation shows that $F_0$ is $D_4$-equivariant and has exactly $8=\lvert D_4\rvert$ zeros in $SO(3)$  given by the orbit under $D_4$ of the identity matrix $I=(e_1, e_2, e_3)$, where $(e_1, e_2,e_3)$ is the canonical basis.
Furthermore, the zeros of $F_0$ are transversal to $SO(3)$. To see this, consider the space $SO(3)$ as a smooth manifold in the space $M_{3\times 3}\simeq\R^9$ of $3\times 3$ matrices.
For $i=1,2,3$, let $R_i(\theta)$ be the rotation in $\R^3$ of angle $\theta$ around the vector $e_i$. For example, the matrix corresponding to $i=1$ is of the form
$$R_1(\theta)=\begin{pmatrix}
1 & 0 & 0\\
0 & \cos(\theta) & -\sin(\theta)\\
0 & \sin(\theta) & \cos(\theta) 
\end{pmatrix}.$$
The vectors $v_i=\frac{d R_i}{d t}(0)$ generate the tangent  space to $SO(3)$ at  $I$. Let $D F_0$ be the derivative of $F_0$ at $I$, then it can be verified that $\abs{\det (D F_0 \cdot (v_1,v_2,v_3))}=2\neq 0$ which implies the transversality at $I$. The transversality at the remaining $7$ zeros follows immediately from the $D_4$-equivariance of $F_0$.

The result now follows directly from Theorem 2.1 in \cite{klartag}. The idea of this theorem can be traced back to Brouwer and was used in the equivariant setting by B\'ar\'any to give an elegant proof of the Borsuk-Ulam theorem. B\'ar\'any's proof is explained in the piecewise linear category in Section 2.2 of \cite{Matousek}. For the reader's convenience we give a sketch of the proof of Theorem 2.1 in \cite{klartag} in our case.

Consider the continuous $D_4$-equivariant function defined on $SO(3)\times [0,1]$ by
\[F(U,t):=(1-t) F_0(U)+t F_1(U).\] 
We approximate $F$ by a smooth $D_4$-equivariant function $F_\varepsilon$ such that $F_\varepsilon(U,0)=F(U,0)=F_0(U)$, $\sup_{U,t}\abs{F(U,t)-F_\varepsilon (U,t)}<\varepsilon$ and $0$ is a regular value of $F_\varepsilon$. The existence of such a smooth equivariant function follows from Thom's transversality theorem \cite{Thom1954} (see also \cite[pp. 68--69]{guillemin2010}), an elementary direct proof can be found in Section 2 of \cite{klartag}. The implicit function theorem implies that $Z_\varepsilon=F_\varepsilon^{-1}(0,0,0)$ is a one dimensional smooth submanifold of $SO(3)\times [0,1]$ on which $D_4$ acts freely. The submanifold $Z_\varepsilon$ is a union of connected components which are diffeomorphic either to an interval, or to a circle, the former having their boundary on $SO(3)\times \{0,1\}$.
The set $Z_\varepsilon$ has an odd number (1) of orbits under $D_4$ intersecting $SO(3)\times \{0\}$. Denote by $\alpha \colon [0,1] \to SO(3)\times [0,1]$ a topological interval of $F_\varepsilon^{-1}(0)$. Let $g\in D_4$, observe that $g(\alpha(0))\neq \alpha(1)$, indeed, if that is the case then $g$ maps $\alpha([0,1])$ to itself and hence has a fixed point, but this would imply that the action of $D_4$ is not free which is a contradiction. We conclude that an odd number of orbits of $Z_\varepsilon$ must intersect $SO(3)\times \{1\}$, i.e. there exists $U_\varepsilon\in SO(3)$ such that $F_\varepsilon(U_\varepsilon,1)=0$. Since the previous discussion holds for every $\varepsilon$, there exists $U\in SO(3)$ such that $F(U,1)=F_1(U)=0$. 
\end{proof}
\begin{remark} 
Let us restate the punch line of the above argument in algebraic topology language: $F_\varepsilon^{-1}(0)\cap SO(3)\times \{0\}$ is a non-trivial $0$-dimensional homology class of $SO(3)$ in the $D_4$-equivariant homology with $Z_2$ coefficients, on the other hand $F_\varepsilon^{-1}(0)$ is a $D_4$-equivariant bordism so $F_\varepsilon^{-1}(0)\cap SO(3)\times \{1\}$ must also be non-trivial in this equivariant homology, and in particular, non empty.
\end{remark}

\begin{remark}
Theorem \ref{thm:equipart} can also be shown using obstruction theory with the aide of a $D_4$-equivariant CW-decomposition of $SO(3)$ (see \cite{BLAG}). 
\end{remark}

\begin{remark}\label{rm:linear}  We shall say for the rest of the paper that $K$ is {\em equipartitioned  by the standard orthonormal basis} $(e_1, e_2, e_3)$ if $(\pm e_1, \pm e_2, \pm e_3) \in \partial K$ and the planes $H_i=e_i^\perp$ satisfy the conditions of Theorem \ref{thm:equipart}. Applying Theorem \ref{thm:equipart} it follows that for every convex, symmetric  body $K \subset \R^3$ there exists $T \in GL(3)$  such that $TK$ is equipartitioned by the standard orthonormal basis.
\end{remark}

\section{Symmetric Mahler conjecture in dimension 3}\label{sec:main}

For a sufficiently regular oriented hypersurface $A \subset \R^n$, define the vector 
 \[V(A)=\int_A n_A(x)dH(x).\]
where $H$ is the $(n-1)$-dimensional Hausdorff measure on $\R^n$ such that $|A|=H(A),$ for $A$ contained in a hyperplane,  and $n_A(x)$ denotes the unit normal to $A$ at $x$ defined by its orientation.
Notice that for $n=3$ one has
\begin{equation}\label{wedge}
V(A):=\left(\int_{A}dx_2\wedge dx_3,\int_{A} dx_1 \wedge dx_3,\int_{A} dx_1 \wedge dx_2 \right),
\end{equation}
because of the following equality between vector valued differential forms
\begin{equation}
n_A(x) dH(x)=(dx_2\wedge dx_3, dx_1\wedge dx_3, dx_1\wedge dx_2). \label{identity}
\end{equation}
Indeed let $T_x$ be the tangent plane at $x$, for a pair of tangent vectors $u,v \in T_x$, $dS(x)(u,v)$ is the signed area of the parallelogram spanned by $u$ and $v$. Let $\theta$ be the angle of intersection between $T_x$ and $e_i^\perp$, and observe that $n_A(x)_i=\cos(\theta)$.
On the other hand since the form $dx_j \wedge d x_k$ doesn't depend on the value of $x_i$, we have \[dx_j\wedge dx_k(u,v)=dx_j\wedge dx_k(P_{e_i}u,P_{e_i}v)=\det(P_{e_i}u,P_{e_i}v).\] 
This is the signed area of the projection of the oriented parallelogram spanned by $u$ and $v$ on  the coordinate hyperplane $e_i^\perp$. Thales theorem implies  \[dx_j\wedge dx_k(P_{e_i}u,P_{e_i}v)= \cos(\theta) dH(x)(u,v)= (n_A(x))_i dH(x)(u,v),\]
establishing identity (\ref{identity}) above. 

For any convex body $K$ in $\R^n$ containing $0$ in its interior, the orientation of a subset $A$ of $\partial K$ is given by exterior normal $n_K$ to $K$ so that $V(A)=\int_A n_K(x)dH(x)$; we define $\C(A):=\{rx; 0\le r\le1, x\in A\}$, and observe that
\[\abs{\C(A)}=\frac{1}{n}\int_A \langle x,n_K(x)\rangle  dH(x).\]

If $K_1$ and $K_2$ are sufficiently regular Jordan regions of $\R^n$ and $K_1\cap K_2$ is an  hypersurface, we define 
$K_1\overrightarrow{\cap} K_2$ 
to be oriented according to the outer normal of $\partial K_1$. So even though $K_1\overrightarrow{\cap} K_2$ equals $K_2 \overrightarrow{\cap} K_1$ as sets, they may have opposite orientations. We denote by $[a,b]= \{ (1-t)a+tb: t \in [0,1]\}$ the segment joining $a$ to $b$. The inequality of the following proposition generalizes inequality (3) in \cite{Meyer} and was proved in \cite{IS} Proposition 3.2.
\begin{proposition}\label{ineq}
Let $K$ be a symmetric convex body in  $\R^n$. Let $A$ be a  Borel subset of $\partial K$ with $|\C(A)| \not=0$, then 
\[
\frac{1}{n}\langle x, V(A)\rangle \leq \abs{\C(A)},\ \forall x\in K.
 \]
So $\frac{V(A)}{n|\C(A)|}\in K^\circ$ and if moreover for some $x_0\in K$ one has $\langle x_0,\frac{V(A)}{n|\C(A)|}\rangle=1$ then $x_0\in\partial K$,  
$\frac{V(A)}{n|\C(A)|}\in \partial K^\circ$ and $[x,x_0]\subset\partial K$, for almost all $x\in A$.
\end{proposition}

\begin{proof}
For any $x\in K$, we have $\langle x,n_K(z)\rangle \le \langle z,n_K(z)\rangle$ for any $z\in\partial K$. Thus 
\[ 
\langle x, V(A)\rangle = \int_A \langle x,n_K(z)\rangle  dH(z)\leq \int_A \langle z,n_K(z)\rangle  dH(z) =n|\C(A)|, \ \forall x\in K. 
\]
It follows that $\frac{V(A)}{n|\C(A)|}\in K^\circ$.
If for some $x_0\in K$ one has $\langle x_0,\frac{V(A)}{n|\C(A)|} \rangle=1$, then clearly $x_0\in\partial K$ and $\frac{V(A)}{n|\C(A)|}\in \partial K^\circ$. Moreover, for almost all $x\in A$ one has $\langle x_0,n_K(x)\rangle=\langle x,n_K(x)\rangle$  thus $[x, x_0] \subset  \partial K$. 

\end{proof}

Using Proposition \ref{ineq} twice we obtain the following corollary.

\begin{corollary} \label{coro} 
Let $K$ be a symmetric convex body in $\R^n$. Consider two Borel subsets $A \subset \partial K$ and $B\subset \partial K^\circ$ such that $|\C(A)|>0$ and $|\C(B)|>0$. Then one has 
 \[\langle V(A), V(B) \rangle  \le n^2|\C(A)||\C(B)|.\]
If there is equality then $[a,\frac{V(B)}{n|\C(B)|}]\subset\partial K$ and $[b,\frac{V(A)}{n|\C(A)|}]\subset \partial K^\circ$ for almost all $a\in A$ and $b\in B$.

\end{corollary}

\begin{proof}[Proof of  the inequality of Theorem \ref{thm:main}]
Since the volume product is continuous, the inequality for an arbitrary convex body follows by approximation by centrally symmetric smooth strictly convex bodies (see \cite{Sc} section 3.4). 
Since the volume product is  linearly invariant, we may assume, using Remark \ref{rm:linear} that  $K$ is  equipartitioned  by the standard  orthonormal basis.
For $\eps\in\{-1;1\}^3$ and  any set $L\subset\R^3$ we define $L(\eps)$ to be the intersection of $L$ with the $\eps$-octant:
\[ 
L(\eps)=\{x\in L; \eps_i x_i\ge0;\    i=1,2,3\}.
\]
From the equipartition of volumes one has $\abs{K(\eps)}=\abs{K}/8$ for every $\eps \in \{-1,1\}^3$.
For $\eps \in\{-1,1\}^3$ let $N(\eps):=\{\eps'\in \{-1,1\}^3: |\eps-\eps'|=2\}$. In other words $\eps' \in N(\eps)$ if  $[\eps, \eps']$ is an edge of the cube $[-1,1]^3$. 
Using Stokes theorem we obtain $V(\partial (K(\eps)))=0$ hence 
\begin{equation}\label{eq:sec}
  V((\partial K)(\eps))=-\sum_{\eps' \in N(\eps)} V(K(\eps) \overrightarrow{\cap} K(\eps'))=\sum_{i=1}^3 \frac{|K\cap e^\perp_{i}|}{4} \eps_{i} e_{i}
\end{equation}
where in the last equality we used the equipartition of areas of $K\cap e^\perp_{i}$.

Since $K$ is strictly convex and smooth, there exists a diffeomorphism $\varphi:\partial K\to \partial K^\circ$ such that $\langle \varphi(x),x\rangle=1$. 
We extend $\varphi$ to $\R^3$ by homogeneity of degree one: $\varphi(\lambda x)=\lambda \varphi(x)$, for any $\lambda\geq 0$ and $x\in\partial K$. 

Then  $K^\circ=\bigcup_{\eps}\varphi (K(\eps))$ and $|K^\circ|=\sum_{\eps}|\varphi (K(\eps))|$. 
From the equipartition of volumes one has
\[\abs{K}\abs{K^\circ} =\sum_{\eps } \abs{K} \abs{\varphi (K(\eps))}=8\sum_{\eps} \abs{K(\eps)}\abs{\varphi (K(\eps))}.\]
From Corollary \ref{coro} we deduce that for every $\eps\in\{-1,1\}^3$
\begin{eqnarray}\label{ineq:coro}
\abs{K(\eps)}\abs{\varphi (K(\eps))} \ge \frac{1}{9} \langle  V((\partial K)(\eps)), V(\varphi(\partial K)(\eps)))\rangle. 
\end{eqnarray}
Thus, using \eqref{eq:sec}
\begin{align}\label{eq:meyer}
\abs{K}\abs{K^\circ}&\ge\frac{8}{9}\sum_{\eps}  \langle V( (\partial K)(\eps)),V( \varphi((\partial K)(\eps)))\rangle = 
\frac{8}{9} \sum_\eps\langle  \sum_{i=1}^3 \frac{ |K\cap e^\perp_{i}|}{4} \eps_i e_i,   V( \varphi((\partial K)(\eps)))\rangle \nonumber\\
 &= \frac{8}{9}\sum_{i=1}^3 \frac{ |K\cap e^\perp_{i}|}{4} \langle e_i, \sum_\eps \eps_i  V( \varphi((\partial K)(\eps)))\rangle.
\end{align}
By Stokes theorem $V(\varphi( \partial K(\eps)))=0$, therefore
\[
V(\varphi( (\partial K)(\eps))) =-\sum_{\eps' \in N(\eps)}V(\varphi(K(\eps) \overrightarrow{\cap} K(\eps'))).
\]
Recall that we have chosen orientations so that for every $\eps' \in N(\eps)$
\[
V(\varphi(K(\eps) \overrightarrow{\cap} K(\eps')))=-V(\varphi(K(\eps')\overrightarrow{\cap} K(\eps))).
\]
Substituting in \eqref{eq:meyer} we obtain
\[
\abs{K}\abs{K^\circ}\ge\frac{8}{9}\sum_{i=1}^3 \frac{ |K\cap e^\perp_{i}|}{4} \langle e_i, \sum_\eps \eps_i  \sum_{\eps'\in N(\eps)}V(\varphi(K(\eps')\overrightarrow{\cap} K(\eps))) \rangle.
\]
If  $[\eps,\eps']$ is an edge of the cube $[-1,1]^3$ let $c(\eps,\eps')$ be the coordinate in which $\eps$ and $\eps'$ differ. For every $i=1,2,3$ we have
\begin{align}\label{eq:punch}
\sum_\eps \eps_i  \sum_{\eps'\in N(\eps)}V(\varphi(K(\eps')\overrightarrow{\cap} K(\eps))) = \sum_\eps  \sum_{\eps'\in N(\eps): c(\eps,\eps')=i}\eps_i V(\varphi(K(\eps')\overrightarrow{\cap}  K(\eps)))  \\
+ \sum_\eps \sum_{\eps'\in N(\eps): c(\eps,\eps')\neq i}\eps_i  V(\varphi(K(\eps')\overrightarrow{\cap}  K(\eps))). \nonumber
\end{align}
The first part of the right hand side of \eqref{eq:punch} can be rewritten as a sum of terms of the form
\[  \eps_i V(\varphi(K(\eps')\overrightarrow{\cap}  K(\eps)))  + \eps_i' V(\varphi(K(\eps)\overrightarrow{\cap}  K(\eps'))) =2\eps_i V(\varphi(K(\eps')\overrightarrow{\cap}  K(\eps))),\]
since $\eps_i=-\eps_i'$. Thus for each $i$, the first part of the right hand side of \eqref{eq:punch} equals
\[\sum_{|\eps-\eps'|=2, \eps_i=1,\eps'_i=-1} 2V(\varphi(K(\eps')\overrightarrow{\cap}  K(\eps))=2V(\varphi(K \cap e_i^\perp)),\] 
where $K\cap e_i^\bot$ is oriented in the direction of $e_i$.
The second part of the sum \eqref{eq:punch} can be rewritten  as a sum of terms of the form
\[\eps_i V(\varphi(K(\eps')\overrightarrow{\cap}  K(\eps))) + \eps_i' V(\varphi(K(\eps)\overrightarrow{\cap}  K(\eps')))=0,\]
since in this case $\eps_i=\eps_i'$.
Let  $P_i$ be the orthogonal projection on the plane $e_i^\bot$. 
Then $P_i: \varphi(K\cap e_i^\bot)\to P_i (K^\circ)$ is a bijection thus from equation (\ref{wedge}) we get
\begin{equation}\label{eq:scalar}
\langle V (\varphi(K\cap e_i^\bot)), e_i\rangle=|P_i(K^\circ)|.
\end{equation}

Since the polar of a section is the projection of the polar, $P_i(K^\circ)=P_i(\varphi(K\cap e_i^\bot))=(K\cap e_i^\bot)^{\circ}$, and we obtain
\begin{eqnarray}\label{ineq:end}
\abs{K}\abs{K^\circ}\ge \frac{4}{9}\sum_{i=1}^3  |K\cap e^\perp_{i}| |P_i(K^\circ)|\ge \frac{4}{9}\cdot 3\cdot \frac{4^2}{2}=\frac{4^3}{3!}=\frac{32}{3},
\end{eqnarray}
where we used the $2$-dimensional Mahler inequality
\begin{eqnarray}\label{ineq:mahler2d}
\abs{K\cap e_i^\bot}\abs{P_i(K^\circ)}\ge \frac {4^2}{2}=8.
\end{eqnarray}
\end{proof}

\begin{remark}
In higher dimensions an equipartition result is not at our disposal, but the generalization of the rest of the proof is straightforward and provides a new large family of examples for which the Mahler conjecture holds.
\end{remark}

\begin{proposition}
If $K \subset \R^n$ is a centrally symmetric convex body that can be partitioned with hyperplanes $H_1,H_2\ldots H_n$ into $2^n$ pieces of the same volume such that each section $K\cap H_i$ satisfies the Mahler conjecture and is partitioned into $2^{n-1}$ regions of the same $(n-1)$-dimensional volume by the remaining hyperplanes, then
\[\abs{K}\abs{K^\circ}\ge \frac{4^n}{n!}.\]
\end{proposition}

The proof is the same, the first inequality has a $\frac{2^n}{n^2}$ factor in front. This time there are $2^{n-1}$ parts on each section and each one appears twice so we multiply by a factor of $\frac{1}{2^{n-2}}$ and the sum has $n$ terms, so the induction step introduces a factor of $\frac{4}{n}$ as desired.

\section{Equality case}\label{sec:equality}
In this section we prove that the only symmetric three dimensional convex bodies that achieve equality in the Mahler conjecture are linear images of the cube and of the cross polytope.

The strategy is to look at the steps in the proof of Theorem \ref{thm:main} where inequalities were used. Specifically, the analogous theorem in dimension $2$ implies that the coordinate sections of $K$ satisfy the Mahler conjecture and therefore are parallelepipeds. Combinatorial analysis of how these sections can interact with the equipartition yields the equality case. The major ingredient of the analysis on how these sections can coexist in a convex body is corollary \ref{coro}. At first one might think that the situation is extremely rigid and there are just a few evident ways in which a convex body with coordinate square sections might be the union of $8$ cones. In fact there is a large family of positions of the cube for which it is equipartitioned by the standard orthogonal basis, and more than one way of positioning the octahedra so that it is equipartitioned by the standard orthogonal basis, as it will become evident in the proof.\\

First we show that the symmetric convex bodies which are equipartitioned by the standard  orthonormal  basis (see Remark \ref{rm:linear})  are uniformly bounded. 
We use the notation $\conv(\cdot)$ to denote the convex hull.

\begin{lemma}\label{lm:bounded} Let $L \subset \R^2$ be a symmetric convex body equipartitioned by the standard orthonormal basis. Then
\begin{equation}\label{eq:r2bound}
B_1^2 \subset L \subset \sqrt{2} B_2^2.
\end{equation}
Let $K \subset \R^3$ be a symmetric convex body equipartitioned by the standard orthonormal  basis. Then 
$$
B_1^3 \subset K \subset 54\sqrt{2} B_1^3.
$$
\end{lemma}

\begin{proof} The first inclusion follows from the fact $\pm e_1, \pm e_2 \in \partial L$.  Our goal is to see how far any point in $L$ may be situated from $B_1^2$. Consider a point $a
=a_1e_1+a_2e_2 \in L \cap (\R_+)^2 \setminus B_1^2$. Note that $|a_1-a_2| \le 1$, since otherwise
$e_j$ would be in the interior of $\conv (a, \pm e_i)$ for $i\not=j$, which would contradict the equipartition assumption. It follows that
$
|L \cap (\R_+)^2| \ge|\conv(0,e_1,e_2,a)|= \frac{a_1+a_2}{2}. 
$
Let $b$ be the intersection of the line through $a$ and $e_1$ and  the line through $-a$ and $-e_2$.  Then
$
L \cap \{ x_1 \ge 0, x_2 \le 0\} \subset \conv(0, b, e_1, -e_2),
$
Indeed, otherwise one of the points $e_1$ or $-e_2$ would not be on the boundary of $L$.  Using that  $|\conv(0, b, e_1, -e_2)|=\frac{b_1-b_2}{2}$ and the equipartition property of $L$, a direct computation shows that
$(a_1-\frac{1}{2})^2+  (a_2-\frac{1}{2})^2 \le \frac{1}{2}.$
Repeating this observation for all quadrants of $L$  we get the result.
 
Now  consider a symmetric convex body $K \subset \R^3$ equipartitioned by the standard orthonormal basis. Let us prove that 
\begin{equation}\label{eq:one-third}
\max_{a\in K(\eps)}\|a\|_1\ge\max_{a\in K}\|a\|_1^{1/3},\quad \forall\eps\in\{-1;1\}^3,
\end{equation}
where $\|a\|_1=\sum |a_i|$. Let $\eps_0\in\{-1;1\}^3$ and $a(\eps_0)\in K(\eps_0)$ such that $\|a(\eps_0)\|_1=\max_{a\in K}\|a\|_1$. Then 
$$
|K(\eps_0)| \ge |\conv(0,e_1,e_2, e_3,a(\eps_0))|=\|a(\eps_0)\|_1/6.
$$
For any $\eps\in\{-1;1\}^3$, one has $K(\eps)\subset\max_{a\in K(\eps)}\|a\|_1B_1^3(\eps)$ thus $|K(\eps)|\le\max_{a\in K(\eps)}\|a\|_1^3/6$, this together with the equipartition property of $K$ gives (\ref{eq:one-third}).

Let $R=\max_{a\in K}\|a\|_1$. Then, from (\ref{eq:one-third}), for all $\eps$ there exists $a(\eps)=a_1(\eps)e_1+a_2(\eps)e_2+a_3(\eps) e_3 \in K(\eps)$ such that $\|a(\eps)\|_1 \ge R^{1/3}$. Note that $\|a(\eps)\|_\infty=\max_i |a_i(\eps)| \ge R^{1/3}/3$. 
Moreover, there exists $\eps \not =\eps'$,  and  $i\in \{1,2,3\}$ such that
\begin{equation}\label{eq:comb}
 |a_i(\eps)|=\|a(\eps)\|_\infty \ge \frac{R^{1/3}}{3},\quad   |a_i(\eps')|=\|a(\eps')\|_\infty \ge \frac{R^{1/3}}{3} \quad\mbox{and}\quad  \sign a_i(\eps) = \sign a_i(\eps').
\end{equation}
Indeed,  among the $8$ vectors $a(\eps)$,  at least three  achieve their $\ell_\infty$-norm at the same coordinate $i$ and,  among these three vectors, at least two have this coordinate of the same sign. 

Consider $a(\eps)$ and $a(\eps')$ as in (\ref{eq:comb}), then $\lambda a(\eps) +(1-\lambda) a(\eps') \in K$ for all $\lambda \in [0,1]$. Since $\eps \not =\eps'$,  there exists a coordinate $j \not=i$ such that either $a_j(\eps)=a_j(\eps')=0$  or $\sign a_j(\eps) \not =\sign a_j(\eps')$.  For some $\lambda \in [0,1]$,  one has
$
\lambda a_j(\eps) +(1-\lambda) a_j(\eps') =0$  and thus    $\lambda a(\eps) +(1-\lambda) a(\eps') \in K \cap e_j^\perp.
$
Using the equipartition of $L:=K \cap e_j^\perp$  and (\ref{eq:r2bound}), together with the properties of $a(\eps)$ and $a(\eps')$,  we get
$$
\frac{R^{1/3}}{3} \le  | (\lambda a_i(\eps) +(1-\lambda) a_i(\eps'))| \le \sqrt{2}.
$$
\end{proof}

Using Lemma \ref{lm:bounded}, we prove the following approximation lemma.

\begin{lemma}\label{approx} Let $L\subset\R^3$ be a symmetric convex body. Then there exists a sequence $(L_m)_m$ of smooth symmetric strictly convex bodies which converges to $L$ in Hausdorff distance and a sequence $(T_m)_m$ of linear invertible maps which converges to a linear invertible map $T$ such that for any $m$ the bodies $T_mL_m$ and $TL$ are equipartitioned by the standard orthonormal basis.
\end{lemma}

\begin{proof}
From  \cite{Sc} section 3.4 there exists a sequence $(L_m)_m$ of smooth symmetric strictly convex bodies converging to $L$ in Hausdorff distance. From Theorem \ref{thm:equipart} there exists a linear map $T_m$ such that $T_mL_m$ is equipartitioned by the standard orthonormal basis. Since $(L_m)_m$  is a convergent sequence it is also  a bounded sequence and  thus there exist constants $c_1, c_2>0$  such that  $c_1B_2^3 \subset L_m \subset c_2B_2^3$. Moreover,  from Lemma \ref{lm:bounded} there exist $c'_1, c'_2>0$ such that $ c'_1B_2^3 \subset T_m L_m \subset c'_2B_2^3$. Thus  $\sup_m\{\|T_m\|, \|T^{-1}_m\|\}<\infty$ and we may select a subsequence $T_{m_k}$ which converges to some invertible linear map $T$. Then $T_{k_m}L_{k_m}$ converges to $TL$.  By continuity, $TL$ is also equipartitioned by the standard orthonormal basis.
\end{proof}

We also need the following lemma.

\begin{lemma}\label{lm:square} Let $P$ be a centrally symmetric parallelogram with non empty interior. Assume that $P$ is equipartitioned by the standard orthonormal axes
  then $P$ is a square.
\end{lemma}

\begin{proof} For some invertible linear map $S$, one has   $S(P)=B_\infty^2$ and the lines $\{t Se_1: t \in \R\}$, $\{t Se_2: t \in \R\}$  equipartition $B_\infty^3$. The square $B_\infty^2$ remains invariant when we apply a rotation of angle $\pi/2$. The rotated lines must also equipartition $B_\infty^2$, this implies that the rotated lines are also invariant since otherwise one cone of $1/4$ of the area would be contained in another cone of area $1/4$. Moreover  $Se_1,Se_2\in\partial B_\infty^2$, thus  $\|Se_1\|_2=\|Se_2\|_2:=\lambda$, and we can conclude that   $\frac 1 \lambda S$ is an isometry and $P$ is a square. 
\end{proof}
 
We shall use the following easy lemma. 
  
 \begin{lemma}\label{observation} Let $K$ be a convex body in $\R^3$. 
 If two segments $(p_1,p_2)$ and $[p_3,p_4]$  are included in $\partial K$ and satisfy $(p_1,p_2)\cap [p_3,p_4]\neq\emptyset$, 
 then $\conv(p_1,p_2,p_3, p_4)\subset H\cap K\subset\partial K$, for some supporting plane $H$ of $K$.
 \end{lemma}
 
 \begin{proof}
Let $H$ be a supporting plane of $K$ such that $[p_3,p_4]\subset H$. Denote by $n$ the exterior normal of $H$. Let $x\in(p_1,p_2)\cap [p_3,p_4]$. Then $\langle p_i,n\rangle\le \langle x,n\rangle$ for all $i$. Moreover there exists $0<\la<1$ such that $x=(1-\la)p_1+\la p_2$. Since $\langle x,n\rangle=(1-\la) \langle p_1,n\rangle+\la\langle p_2,n\rangle \le \langle x,n\rangle$ one has $p_1,p_2\in H$. Since $H\cap K$ is convex, we conclude that $\conv(p_1,p_2,p_3, p_4)\subset H\cap K$.
 \end{proof}
 
 \noindent
 \begin{proof}[Proof of the equality case of Theorem 1.]
Let $L$ be a symmetric convex body such that $\abs{L}\abs{L^\circ}=\frac{32}{3}$. Applying Lemma \ref{approx} and denoting $K=TL$ and $K_m=T_mL_m$, we get that the bodies $K_m$ are smooth and strictly convex symmetric, the bodies $K_m$ and $K$ are equipartitioned by the standard orthonormal basis and the sequence $(K_m)_m$ converges to $K$ in Hausdorff distance. 
Applying inequality (\ref{ineq:end}) to  $K_m$ and taking the limit we get 
$$
\frac{32}{3}=\abs{K}\abs{K^\circ}=\lim_{m\to+\infty}|K_m||K_m^\circ|\ge\frac{4}{9}\sum_{i=1}^3 \lim_{m\to+\infty} |K_m\cap e^\perp_{i}| |P_i(K_m^\circ)|= \frac{4}{9}\sum_{i=1}^3  |K\cap e^\perp_{i}| |P_i(K^\circ)|\ge \frac{32}{3},
$$
where we used (\ref{ineq:mahler2d}). Thus we deduce that  for all $i=1,2,3$,
\[
\abs{K\cap e_i^\bot}\abs{P_i(K^\circ)}= 8.
\]
Reisner \cite{Reisner} and Meyer \cite{Meyer} showed that if Mahler's equality in $\R^2$ is achieved then the corresponding planar convex body is a parallelogram. It follows that the sections of $K$ by the planes $e_i^\bot$ are parallelograms, which are equipartitioned by the standard orthonormal axes. From  Lemma \ref{lm:square} the  coordinate sections $K\cap e_i^\bot$ are squares and one may write 
\[K\cap e_i^\bot=\conv(a_\omega;\ \omega_i=0;\ \omega_j\in\{-;+\}, \forall j\neq i),\quad\hbox{where}\quad a_\omega\in\sum_{i=1}^3\R_{\omega_i}e_i.
\] 
For example 
$K\cap e_3^\bot=\conv(a_{++0}, a_{+ - 0}, a_{-+0}, a_{--0}).$
We discuss the four different cases, which also appeared in \cite{IS}, depending on  the location of the vertices of $K\cap e_i^\bot$.\\

{\bf Case 1.} If exactly one of the coordinate sections of $K$ is the canonical $B_1^2$ ball, then  $K$ is a parallelepiped. \\ 
For instance, suppose that for
  $i=1,2$, $K\cap e_i^\bot\neq\conv(\pm e_j; j\neq i)$, but $K\cap e_3^\bot=\conv(\pm e_1, \pm e_2)$.  We apply Lemma \ref{observation} to the segments  $[a_{+0+}, a_{-0+}]$ and $[a_{0++}, a_{0-+}]$
to get that the supporting plane $S_3$ of $K$ at $e_3$ contains the quadrilateral $F_3:=\conv(a_{+0+}, a_{0++}, a_{-0+}, a_{0-+})$ and $e_3$ is in the relative interior of the face $K\cap S_3$ of $K$. Moreover $(a_{+0+}, a_{+0-})$ and $[e_1,e_2]$ are included in $\partial K$ and intersect at $e_1$. Thus from Lemma \ref{observation} the triangle $\conv(a_{+0+}, a_{+0-},e_2)$ is included in a face of $K$. In the same way, this face also contains the triangle $\conv(a_{0+-}, a_{0++},e_1)$. Thus the quadrilateral $\conv(a_{+0+}, a_{+0-}, a_{0+-}, a_{0++})$ is included in a face of $K$. In the same way, we get 
that the quadrilateral $\conv(a_{0+-}, a_{0++}, a_{-0+}, a_{-0-})$ is included in a face of $K$. We conclude that $K$ is a symmetric body with $6$ faces, thus a parallelepiped.\\

{\bf Case  2.} If exactly two of the coordinate sections of $K$ are canonical $B_1^2$ balls, then $K=B_1^3$. \\
For instance, suppose that 
 for every $i=1,2$, $K\cap e_i^\bot=\conv(\pm e_j; j\neq i)$, but $K\cap e_3^\bot\neq\conv(\pm e_1, \pm e_2)$.  Then the segments $(a_{++0}, a_{+-0})$ and $[e_1,e_3]$ are included in $\partial K$ and intersect at $e_1$. Thus from Lemma \ref{observation}, the triangle $\conv(a_{++0}, a_{+-0}, e_3)$ is 
in a face of $K$. Using the other octants, we conclude that $K=\conv(e_3,-e_3, K\cap e_3^{\perp})$.\\

{\bf Case  3.} If for every $1\le i\le 3$, $K\cap e_i^\bot\neq\conv(\pm e_j; j\neq i)$, then $K$ is a cube. \\
One has  $(a_{++0}, a_{+-0})\subset \partial K$ and $(a_{+0+}, a_{+0-})\subset\partial K$ and  these segments intersect at $e_1$.  From Lemma~\ref{observation}, the supporting plane $S_1$ of $K$ at $e_1$ contains the quadrilateral $\conv(a_{++0}, a_{+-0}, a_{+0+}, a_{+0-})$ and $e_1$ is in the relative interior of the face $K\cap S_1$.   Reproducing this in each octant, 
we conclude that $K$ is contained in the parallelepiped $L$ delimited by the planes $S_i$ and $-S_i$, $1\le i\le3$. Let $u_i\in\partial K^\circ$ be such that $S_i=\{x\in \R^3; \langle x, u_i \rangle =1\}$. It follows that $K^{\circ}$ contains  the octahedron $L^\circ= \conv(\pm u_1, \pm u_1,\pm u_3)$. Observe that $K\cap e_i^\perp= L\cap e_i^\perp$ 
and thus 
$
P_i K^\circ= P_i L^{\circ},
$
 for each  $i=1,2,3$. 
Let $H_i$ the oriented plane spanned by $u_j, u_k$, with the positive side containing $u_i$, where $i, j, k$ are distinct indices, $i,j,k \in \{1,2,3\}$ and let $H_i^{\eps_i}$ be the half-space containing $\eps_i u_i$. Let  
$$
K^{\circ}_\eps =K^{\circ}\cap H_1^{\eps_1} \cap H_2^{\eps_2}\cap H_3^{\eps_3} \mbox{ and } (\partial  K^{\circ})_\eps =\partial K^{\circ}\cap H_1^{\eps_1} \cap H_2^{\eps_2}\cap H_3^{\eps_3}.
$$
  We  note that $P_i(u_j),$  $i\not = j$ is orthogonal to an edge of $K \cap e_i^\perp$, thus
$P_i(u_j)$ is a vertex of $P_i(K^\circ)$, taking into account that $P_i(K^\circ)$ is a square we get 
$$
P_i(K^\circ)=P_i (\conv(\pm u_j, \pm u_k))=P_i(K^\circ \cap H_i),
$$ where $i, j, k$ are distinct indices, $i,j,k \in \{1,2,3\}$. Similarly to (\ref{eq:scalar}),  from equation (\ref{wedge}) we get
\begin{equation}\label{eq:descr}
\langle V(K^\circ \cap H_i), e_i\rangle=|P_i(K^\circ)|,
\end{equation}
where $K^\circ\cap H_i$ is oriented in the direction of $u_i$.
As in the proof 
of Theorem \ref{thm:main},  we  use  the equipartition of $K$, to  get
$$
\frac{32}{3}=|K||K^{\circ}|=8\sum_{\eps} |K(\eps) | |K^{\circ}_\eps|  \ge  \frac{8}{9} \sum \langle  V((\partial K)(\eps)), V((\partial K^{\circ})_\eps)\rangle,
$$
where the last inequality follows  from  Corollary \ref{coro}. Since $K^\circ_\eps$  is a convex body one has 
\[
V((\partial K^{\circ})_\eps) =-\sum_{\eps'\in N(\eps)}V(K^\circ_\eps \overrightarrow{\cap}  K^\circ_{\eps'}).
\]
We thus may continue  as in the proof of Theorem  \ref{thm:main}, we finally get
$$
 \frac{32}{3}=|K||K^{\circ}|=8\sum_{\eps} |K(\eps) | |K^{\circ}_\eps|  \ge  \frac{8}{9} \sum \langle V((\partial K)(\eps)), V((\partial K^{\circ})_\eps)\rangle\ge \frac{4}{9} \sum_{i=1}^3 |K \cap e_i^{\perp}|  |P_i(K^{\circ})| =\frac{32}{3}.
$$
It follows that $\langle \frac{V((\partial K)(\eps))}{3 |K(\eps) |}, \frac{V((\partial K^{\circ})_\eps)}{3 |K^\circ_\eps |}\rangle =1$. Using Corollary \ref{coro} we define the points 
\begin{equation}\label{eq:points}
y_\eps:= \frac{V((\partial K)(\eps))}{3 |K(\eps) |} \in\partial K^{\circ} \mbox{ and   } x_\eps:=\frac{V((\partial K^{\circ})_\eps)}{3 |K^\circ_\eps |}\in \partial K, 
\end{equation}
for all $\eps\in \{-1,1\}^3$.  Since $[e_1,x_{+,+,+}]\subset \partial K$ one has $x_{+++}\in S_1$. In the same way $x_{+++}\in S_2$ and $x_{+++}\in S_3$.  Reproducing this in each octant and for each $x_\eps$, we conclude that $K=L$.\\
  
{\bf Case  4.} If for every $1\le i\le 3$, $K\cap e_i^\bot=\conv(\pm e_j; j\neq i)$, then $B_1^3=K$.\\
 Since $B_1^3\subset K$ one has $K^\circ\subset B_\infty^3$ and $P_i(K^\circ)=P_i( B_\infty^3)$.   
If there exists  $\eps \in \{-1;1\}^3$ such that $\eps \in K^\circ$, then  $K(\eps) =B_1^3(\eps)$ and using equipartition property of $K$ we get $K=B_1^n$. Assume, towards the contradiction,  that $K^\circ \cap   \{-1;1\}^3 =\emptyset$. Since  $P_iK^\circ= [-1,1]^3\cap e_i^\perp$,  
each {\it open} edge $(\varepsilon,\varepsilon')$ of the cube $[-1,1]^3$  contains at least
one point $C_{\varepsilon,\varepsilon'}\in \partial K^\circ$.   
Using the symmetry of $K^\circ$ we may assume that the twelve selected points $C_{\varepsilon,\varepsilon'}$ are pairwise symmetric.  
For each $i=1,2,3,$ consider the $4$ points $C_{\varepsilon,\varepsilon'}$ belonging to edges of $[-1,1]^3$ parallel to the direction $e_i$. They generate a plane $H_i$ passing through the origin. Thus, we have defined $3$ linearly  independent planes   $H_1, H_2$ and $H_3$  (note that linear independence follows from the fact that $H_i$ does not contain the vertices of the cube and passes through edges parallel to $e_i$ only).  Those planes define a partition of $K^\circ$ into $8$ regions with non-empty interior, moreover,
$$
P_i(K^\circ)=P_i(K^\circ \cap H_i).  
$$
We repeat verbatim the construction of Case 3 using the partition of $K^\circ$ by the planes $H_i$ and defining $H_i^{\eps_i}$ to be a half-space containing $\eps_i e_i$. Using Proposition \ref{ineq} and Corollary  \ref{coro}  we get  that  for each point of $y\in \partial K^\circ$, for some $\varepsilon\in \{-1;1\}^3$,  one has
$[y,y_\eps]\subset (\partial K^\circ)_\eps$ (see  (\ref{eq:points}) for the definition of $y_\eps$).  Together with the fact that each face of $B_\infty^3$  intersects $K^\circ$ on a facet of $K^\circ$ this gives that $B_\infty^3=K^\circ$, which contradicts our assumption that 
$K^\circ \cap   \{-1;1\}^3 =\emptyset$.
\end{proof}

\section{Stability}\label{sec:stability}

The goal of this section is to establish the stability of the $3$ dimensional Mahler inequality.  
 The {\it Banach-Mazur distance} between symmetric convex bodies $K$ and $L$ in $\R^n$ is defined as
  $$d_{BM}(K,L) = \inf \{d\ge1 : L\subset TK\subset dL, \mbox{ for some } T\in{\rm GL}(n)\}.$$
 We refer to \cite{AGM} for properties of the Banach-Mazur distance, in particular, for John's theorem, which states that $d_{BM}(K, B_2^n) \le \sqrt{n}$ for any symmetric convex  body $K\subset \R^n$ and thus $d_{BM}(K, L) \le n$ for any pair of symmetric convex  bodies $K,L \subset \R^n$. 
 
\begin{thm}\label{thm:stability_BM}
There exists an absolute constant $C>0$, such that for every symmetric convex body $K \subset \R^3$ and $\delta>0$  satisfying
$\Pp(K) \leq (1+ \delta)\Pp(B_\infty^3)$, one has
$$
\min\{ d_{BM}(K, B_\infty^3), d_{BM}(K, B_1^3)\} \le 1+C\delta.
$$
\end{thm}

We start with a general simple lemma on compact metric spaces and continuous functions.
 \begin{lemma}\label{lem:metric}
  Let $(A,d)$ be a compact metric space, $(B,d')$  a metric space, $f:A\to B$ a continuous function and $D$ a closed subset of $B$. Then, 
   \begin{enumerate}
  \item For any $\beta>0$ there exists $\alpha>0$ such that $d(x,f^{-1}(D))\ge\beta$ implies $d'(f(x),D)\ge\alpha$.
   \item If  there exists  $c_1,c_2>0$ such that $d(x,f^{-1}(D))<c_1$ implies $d'(f(x), D))\ge c_2d(x,f^{-1}(D))$ then  for some $C>0$, one has
$d(x,f^{-1}(D)) \le C d'(f(x), D)), \  \forall x \in A.$
  \end{enumerate}
 \end{lemma}
 \begin{proof} (1) Let $\beta>0$ such that $A_\beta = \{x \in A:  d(x, f^{-1}(D) \ge \beta\}\neq\emptyset$. Then $A_\beta$ is compact and since the function $x\mapsto d'(f(x), D))$ is continuous on $A_\beta$, it reaches its infimum $\alpha$ at some point $x_0\in A_\beta$. We conclude that for all $x\in A_\beta$ one has $d'(f(x), D))\ge\alpha=d'(f(x_0), D))>0$ since $x_0\notin f^{-1}(D)$.\\ 
 (2) Consider two cases. First assume that $d(x,f^{-1}(D))<c_1$, then  $d'(f(x), D) \ge c_2d(x,f^{-1}(D))$ and we select $C=1/c_2$. Now assume
 that $d(x,f^{-1}(D)) \ge c_1$, then using (1), with $\beta=c_1$ we get
 $$
d'(f(x),D)\ge\alpha \ge \frac{\alpha}{{\rm diam}(A)}  d(x,f^{-1}(D)) \  \forall x \in A,
 $$
 where ${\rm diam}(A)$ is the diameter of the metric space $A$ and we conclude with $C=\max\{1/c_2, \frac{{\rm diam}(A)}{\alpha} \}$.
 \end{proof}

\begin{proof}[Proof of Theorem \ref{thm:stability_BM}] The proof follows from Lemma \ref {lem:metric}  together with the stability theorem proved in \cite{NazZva} and equality cases proved in the previous section. Using the linear invariance of the volume product, together with John's theorem we reduce to the case when  $B_2^3\subseteq K \subseteq \sqrt{3} B_2^3$. Our metric space $A$ will be  the set of such bodies  with the Hausdorff metric $d_H$ (see, for example, \cite{Sc}). Let $B=\R$.  Then $f:A \to B$, defined by  $f(K)=\Pp(K)$,  is continuous on $A$ (see for example \cite{FMZ}).  Finally, let $D=\Pp(B_\infty^3)$. Using the description of the equality  cases proved in the previous section we get that
$$
f^{-1}(D)=\{K\in A; \Pp(K)=\Pp(B_\infty^3)\}=\{K\in A; \exists\ S\in {\rm SO}(3); K=S B_\infty^3\ \hbox{or}\ K=\sqrt{3}SB_1^3\}.
$$
Note that $B_\infty^3$ is in John position (see for example \cite{AGM}) and thus  if $ B_2^3 \subset T B_\infty^3 \subset \sqrt{3}B_2^3$ for some $T \in GL(3)$, then $T \in SO(3)$.

We show that the assumptions in the  second part of Lemma \ref{lem:metric} are satisfied. As observed in  \cite{BH} the result of \cite{NazZva} may be stated in the following way: there exists  constants $\delta_0(n), \beta_n>0$, depending on dimension $n$ only,  such that  for every  symmetric convex body $K \subset \R^n$ satisfying $d_{BM}(K, B_\infty^n) \le 1+ \delta_0(n)$ we get
$$
\P(K) \ge  (1+\beta_n (d_{BM}(K, B_\infty^n)-1))\P(B_\infty^n).
$$
Using that $d_{BM}(K^\circ, L^\circ)= d_{BM}(K, L)$,  we may restate the $3$ dimensional version of the above stability theorem in the following form: there are absolute constants $c_1, c_2 >0$ such that  for every  symmetric convex body $K \subset \R^3$ satisfying
$\min\{ d_{BM}(K, B_\infty^3), d_{BM}(K, B_1^3)\} :=1+d \le  1+c_1,$
one has
$$
\P(K) \ge  \P(B_\infty^3)+ c_2 d.
$$
To finish checking the assumption, note that for all $K,L$ convex bodies such that $B_2^3\subseteq K,L \subseteq \sqrt{3} B_2^3$:
\begin{eqnarray}\label{eq:dist}
d_{BM}(K, L)-1 \le   \min_{T\in GL(3)} d_H (TK, L) \le  \sqrt{3}(d_{BM}(K, L) -1 ). 
\end{eqnarray}
Applying Lemma \ref{lem:metric} we deduce that there exists $C>0$ such that for all $K$ such that $B_2^3\subseteq K \subseteq \sqrt{3} B_2^3$:
\[
\min_{S\in SO(3)}\min(d_H(K,SB_\infty^3), d_H(K,S\sqrt{3}B_1^3))\le C|\P(K)-\P(B_\infty^3)|.
\]
Using (\ref{eq:dist}) we conclude the proof.

\end{proof}

\begin{remark} The same method as in the proof of Theorem \ref{thm:stability_BM}, i.e. applying Lemma  \ref{lem:metric}, known equality cases and the results from \cite{Kim, NazZva} can be used to present  shorter proofs of the stability theorems given in \cite{BH, ZvK}.
\end{remark}

\vspace{1mm}

\noindent
{\footnotesize\sc Matthieu Fradelizi:}
  {\footnotesize LAMA, Univ Gustave Eiffel, Univ Paris Est Creteil, CNRS, F-77447 Marne-la-Vallée, France. }\\[-1.3mm]
  {\footnotesize e-mail: {\tt matthieu.fradelizi@univ-eiffel.fr \tt }}\\

\noindent
{\footnotesize\sc Alfredo Hubard:}
  {\footnotesize LIGM, Univ. Gustave Eiffel, CNRS, ESIEE Paris, F-77454 Marne-la-Vallée, France }\\
  {\footnotesize e-mail: {\tt alfredo.hubard@univ-eiffel.fr \tt }}\\

\noindent
{\footnotesize \sc Mathieu Meyer: }
  {\footnotesize LAMA, Univ Gustave Eiffel, Univ Paris Est Creteil, CNRS, F-77447 Marne-la-Vallée, France. }\\
  {\footnotesize e-mail: {\tt mathieu.meyer@univ-eiffel.fr \tt }}\\

\noindent
  {\footnotesize{\sc Edgardo Rold\'an-Pensado:}
  {\footnotesize Centro de Ciencias Matem\'aticas, UNAM-Campus Morelia, Antigua Carretera a P\'atzcuaro \# 8701, Col. Ex Hacienda San Jos\'e de la Huerta, Morelia, Michoac\'an M\'exico, C.P. 58089.\\
  {\footnotesize e-mail: {\tt e.roldan@im.unam.mx \tt }}\\
   
   \noindent
  {\footnotesize{\sc Artem Zvavitch:}
  {\footnotesize Kent State University, Department of Mathematics, Kent, OH 44242, USA }\\
  {\footnotesize e-mail: {\tt  zvavitch@math.kent.edu  \tt }}

\end{document}